\documentclass[11pt, leqno]{article}
\usepackage{amsmath,amssymb,amsthm, epsfig}
\usepackage{hyperref}
\usepackage{amsmath}
\usepackage[pdftex]{color}
\usepackage{mathtools}
\usepackage{color}
\usepackage{ulem}
\usepackage{mathrsfs}
\usepackage{wasysym}
\usepackage{stackrel}
\usepackage{pgfplots}
\pgfplotsset{compat=newest}
\usepackage{comment}

\title{\Large{\bf Hölder regularity up to the boundary for the g-Laplacian on Reifenberg flat domains}}
\author{\it by \smallskip \\ 
Alan Pio Sousa\footnote{\noindent Universidade Federal do Cear\'{a}. Departamento de Matem\'atica. Fortaleza - CE, Brazil. \noindent \texttt{E-mail address: \url{alanpio@ufc.br}}}
}

\newlength{\hchng}
\newlength{\vchng}
\newtheorem{theorem}{Theorem}[section]
\newtheorem{lemma}[theorem]{Lemma}
\newtheorem{proposition}[theorem]{Proposition}
\newtheorem{corollary}[theorem]{Corollary}
\theoremstyle{definition}

\newtheorem{definition}[theorem]{Definition}

\theoremstyle{remark}
\newtheorem{remark}[theorem]{Remark}

\numberwithin{equation}{section}

\newcommand{\defeq}{\mathrel{\mathop:}=}
\newcommand{\intav}[1]{\mathchoice {\mathop{\vrule width 6pt height 3 pt depth  -2.5pt
			\kern -8pt \intop}\nolimits_{\kern -6pt#1}} {\mathop{\vrule width
			5pt height 3  pt depth -2.6pt \kern -6pt \intop}\nolimits_{#1}}
	{\mathop{\vrule width 5pt height 3 pt depth -2.6pt \kern -6pt
			\intop}\nolimits_{#1}} {\mathop{\vrule width 5pt height 3 pt depth
			-2.6pt \kern -6pt \intop}\nolimits_{#1}}}

\begin{document}
\maketitle
	
\begin{abstract}
\noindent We establish boundary Hölder regularity for weak solutions to a class of nonlinear elliptic Dirichlet problems with nonstandard growth  posed on Reifenberg flat domains. More precisely, for any prescribed exponent \(\alpha\in(0,1)\), we show that weak solutions are \(\alpha\)-Hölder continuous up to the boundary provided that the Reifenberg flatness parameter is sufficiently small. The proof combines an iterative boundary decay argument with an ABP-type maximum principle in the Orlicz--Sobolev setting.

\vspace{0.2cm}
\end{abstract}

\section{Introduction and main results}
	
In this manuscript, we investigate boundary regularity for weak solutions to the quasilinear elliptic equation 
\begin{equation}\label{equ1}
    \left\{
    \begin{aligned}
        \Delta_g u &= f \quad \text{in } \Omega, \\
        u &= \psi \quad \text{on } \partial\Omega.
    \end{aligned}
    \right.
\end{equation}
where $\Omega\subset \mathbb{R}^n$ is a bounded domain.
The operator $\Delta_g$ is the so-called g-Laplace operator, defined by
$$\Delta_gu:=\operatorname{div}\left(\frac{g(|\nabla u|)}{|\nabla u|}\nabla u\right),$$
where $g:[0,\infty)\to [0,\infty)$ is the derivative of a Young function G, satisfying the following structural conditions:
\begin{itemize}
    \item {\rm (\textbf{PC})}\,\textit{Primitive condition:}
    $$G'(t)=g(t),\quad g\in C^0([0,+\infty)])\cap C^1((0,+\infty)) $$
    \item {\rm (\textbf{QC})}\,\textit{Quotient condition:}\, There exist positive constants $\delta_0\leq g_0$ such that
    $$ \delta_0\leq \frac{t\cdot g'(t)}{g(t)}\leq g_0,\quad\forall t>0. $$
\end{itemize}
Originally introduced by Lieberman in \cite{Lieb}, these assumptions offer a natural analytical framework for studying boundary regularity in singular or degenerate elliptic equations with non-standard growth. Clearly, this operator recovers the classical p-Laplacian by taking $g(t)=t^{p-1}$ with $\delta_{0}=g_{0}=p-1$, respectively. A different example consists of a function $g(t)=t^a\log(bt +c)$ with $a, b,c>0$ that satisfies {\bf (QC)} for $\delta_0=a$ and $g_0=a+1.$
Moreover, the class of functions satisfying the conditions introduced by Lieberman is broad, since any positive linear combination of functions satisfying \textbf{(PC)} and \textbf{(QC)} also satisfies these conditions. Furthermore, the product and composition of functions satisfying \textbf{(PC)} and \textbf{(QC)} also satisfy these conditions.

Boundary regularity for nonlinear PDEs has long been a central topic in analysis, particularly in elliptic and parabolic problems where the behavior of weak solutions near the boundary presents delicate analytical difficulties. In this context, the geometry of the domain and the regularity of the boundary data are crucial in determining the degree of regularity that can be achieved.

Under suitable regularity assumptions on the domain, weak solutions to the Dirichlet problem \eqref{equ1} are known to enjoy corresponding regularity properties; see, for instance, \cite{BraSou25}. In this paper we investigate boundary regularity under substantially weaker geometric assumptions, namely when the underlying domain is a Reifenberg flat domain. This means that for the domain $\Omega,$ at each boundary point and every scale the boundary of the domain is between two hyperplanes separated by a distance which depends on the scale.

The notion of Reifenberg-flat sets was first introduced by Reifenberg in 1960 \cite{Reif} while he was working on the Plateau problem. Since then, these sets have played an important role in the study of minimal surfaces. Over the past decades, a substantial body of work has been devoted to the study of elliptic and parabolic equations on Reifenberg flat domains. Interest in this setting has continued to grow, reflecting the importance of Reifenberg flatness as a natural geometric framework extending beyond the classical smooth-domain theory. Among the contributions in this direction are the works of Byun, Wang, Zhou and Yang \cite{By1, By2, By3}, Lemenant and Sire \cite{ Leme3}, and Milakis and Toro \cite{Mila} (see also the references therein), which address various regularity and qualitative properties of solutions in Reifenberg flat domains. 

More recently, the works \cite{Wu, Et, Lian} have introduced new approaches to the study of boundary regularity for solutions in Reifenberg flat domains. These developments have led to boundary regularity results, together with quantitative estimates, in a variety of settings, including the Laplace equation, linear elliptic equations in both divergence and non-divergence form, fully nonlinear elliptic equations, the \(p\)-Laplace equation, and the fractional Laplacian. Furthermore, in \cite{Prade}, sharp boundary regularity estimates were established for solutions to nonlocal elliptic equations associated with operators comparable to the fractional Laplacian in Reifenberg flat domains. Our goal is to extend that theory to the framework considered here. We establish an ABP-type estimate for the difference of weak solutions to \eqref{equ1}, which provides the main tool for obtaining the regularity results established in this paper. Our main operator, the \(g\)-Laplacian, displays either degenerate elliptic or singular behavior and, unlike the homogeneous \(p\)-Laplacian, lacks homogeneity. This absence of homogeneity creates additional challenges and forces us to develop new arguments at some stages of the analysis.

We now introduce some notation that establishes the framework for the main results. The first of these is a notion of  Reifenberg flat domain introduced in \cite{Et}.

\begin{definition}
  We say that $\Omega$ is $\delta-$Reifenberg flat from the exterior at $x_0\in\partial\Omega$ if there exists $r_0>0$ such that the following holds: for any $0<r<r_0,$ there exists a coordinate system $\{y_1,...,y_n\}$ (isometric to the original coordinate system) such that $x_0=0$ in this coordinate system and
  $$B_r(0)\cap\Omega\subset B_r(0)\cap\{y_n>-\delta r\}. $$
\end{definition}

\begin{remark}
 We say that $\Omega$ satisfies the exterior cone condition with
slope $\delta$ at $x_0\in\partial\Omega$ if there exist $r_0>0$ and a unit vector
$\mathbf{n}$ such that
\[
\left\{
x\in B(x_0,r_0): (x-x_0)\cdot\mathbf{n}
<-\delta|x-x_0|
\right\}
\subset \Omega^c.
\]
Clearly any domain satisfying the exterior cone condition with slope \(\delta\) is \(\delta\)-Reifenberg flat from the exterior. Consequently, all results established under the assumption that \(\Omega\) is \(\delta\)-Reifenberg flat from the exterior at 0 remain valid if this condition is replaced by the exterior cone condition at 0 with slope \(\delta\). In particular, these results apply whenever \(\partial\Omega\) is \(C^1\) at 0.
    
\end{remark}

We now introduce the appropriate spaces for the boundary data.

\begin{definition}
    Let $\Omega\subset\mathbb{R}^n$ be a bounded measurable set and $f\in L^\infty(\Omega).$ We say that $f$ is $C^\alpha\,(0<\alpha\leq 1)$ at $x_0\in \Omega$ or $f\in C^\alpha(x_0)$ if there exist $r_0>0$ and $K>0$ such that
    $$|f(x)-f(x_0)|\leq K|x-x_0|^\alpha,\quad \forall x\in \Omega\cap B_{r_0}(x_0).$$

  We say $f\in C_q^{-p,\alpha}$ at $x_0$ or $f\in C_q^{-p,\alpha}(x_0)$ where $p>0, q\geq 1, \alpha >0,$ if there exist $r_0>0$ and $K>0$ such that
   \begin{equation}\label{equ2}
       ||f||_{L^q(\Omega\cap B_r(x_0))} \leq Kr^{\alpha-p+n/q},\quad\forall\, 0<r<r_0.
   \end{equation}
  We set
   $$||f||_{C_q^{-p,\alpha}(x_0, r_0)}=\min\{K; \text{(\ref{equ2}) holds with}\,\, K \}. $$
 If $f\in C_q^{-p,\alpha}(x)$ for any $x\in \Omega$ with the same $r_0$ and 
  $$ \sup_{x\in \Omega}||f||_{C_q^{-p,\alpha}(x,r_0)}<+\infty, $$
    we say that $f\in C_q^{-p,\alpha}(\Omega).$
\end{definition}

Our first result concerns Hölder regularity up to the boundary for g-harmonic functions.

\begin{theorem}\label{thm1}
    Let $G$ be an N-function satisfying {\bf (PC)-(QC)} and $\Omega\subset \mathbb{R}^n$ a bounded domain. Furthermore, let $u\in W^{1,G}(\Omega)\cap L^\infty(\Omega)$ be a weak solution of 
     \[
     \begin{cases}
     \Delta_g u = 0, & \text{in } \Omega \cap B_1 \\[4pt]
      u = \psi, & \text{on } \partial\Omega \cap B_1
      \end{cases}
     \]
     where $\psi\in C^\alpha(0)$ for some $0<\alpha<1.$ Suppose that $\Omega$ is $\delta-$Reifenberg flat from the exterior at $0,$ for some $0<\delta<1$ depend on $\delta_0, g_0, \alpha$ and $n.$ Then $u$ is $C^\alpha$ at $0$ and for any $x\in \Omega\cap B_1,$
     $$|u(x)-u(0)| \leq C|x|^\alpha\left(||u||_{L^\infty(\Omega\cap B_1)}+ [\psi]_{C^\alpha(0)} \right), $$
     where $C>0$ depends on $n, g_0, \delta_0$ and $\alpha.$
\end{theorem}

 For the next result, we assume that $1<\delta_0\leq g_0$ and that $g$ is convex. This last assumption parallels the degenerate setting of the homogeneous problem, which corresponds to the case $p\geq 2.$

\begin{theorem}\label{thm2}
    Let $G$ be an N-function satisfying {\bf (PC)-(QC)} with $1<\delta_0\leq g_0$ and $g$ convex function. Furthermore, let $\Omega\subset \mathbb{R}^n$ a bounded domain and $u\in W^{1,G}(\Omega)\cap L^\infty(\Omega)$ be a weak solution of  
    \[
    \begin{cases}
    \Delta_g u = f, & \text{in } \Omega \cap B_1 \\
    u = \psi, & \text{on } \partial\Omega \cap B_1
    \end{cases}
    \]
    
     where $\psi\in C^\alpha(0)$ and $f\in C_q^{-n/q,(1+g_0)\alpha}$ for some $0<\alpha<1$ and $q>n.$ Suppose that $\Omega$ is $\delta-$Reifenberg flat from the exterior at $0,$ for some $0<\delta<1$ depend on $\delta_0, g_0, \alpha$ and $n.$ Then $u$ is $C^\alpha$ at $0$ and for any $x\in \Omega\cap B_1,$
     $$|u(x)-u(0)| \leq C|x|^\alpha\left(||u||_{L^\infty(\Omega\cap B_1)}+ ||f||_{C_q^{-n/q,(1+g_0)\alpha}(0)}+ [\psi]_{C^\alpha(0)} +1 \right)^{\left(1+\frac{1+1/g_0}{1+\delta_0}\right)(1+g_0)^{-1}+1}, $$
     where $C>0$ depends on $n, g_0, \delta_0$ and $\alpha.$
\end{theorem}

The global Hölder regularity now follows by combining the boundary estimates obtained above with the interior regularity theory developed in \cite{Lieb}. As the argument is standard, we refrain from providing the details.
\begin{corollary}
Assume that \(\Omega\) is \(\delta\)-Reifenberg flat from the exterior at every point \(x_0\in\partial\Omega\), and that \(\psi\in C^\alpha(\partial\Omega)\). Let \(u\) be a weak solution of \eqref{equ1} corresponding to an N-function G satisfying the assumptions of either Theorem \ref{thm1} or Theorem \ref{thm2} for
$f\in C_q^{-n/q,(1+g_0)\alpha}(\Omega)$ with the same radius \(r_0\). Then
$u\in C^\beta(\overline{\Omega})$ for some \(\beta\leq \alpha\).
\end{corollary}


\section{Preliminary tools}\label{Sec2}

In this section, we collect some auxiliary results that will be used throughout the paper. 
Associated with the $N$-function $G$, we consider the Orlicz and Orlicz--Sobolev spaces.

\begin{definition}
Given an $N$-function $G$, we define the Orlicz space $L^{G}(\Omega)$ as the set of measurable functions $h$ in $\Omega$ such that
\[
\int_{\Omega}G(|h(x)|)\,dx<\infty,
\]
where we adopt the Luxemburg norm
\[
\|h\|_{L^{G}(\Omega)}\defeq\inf\left\{\lambda>0:\int_\Omega G\left(\frac{|h(x)|}{\lambda}\right)\,dx\leq 1\right\}.
\]
Associated with the Orlicz space, we define the Orlicz--Sobolev space $W^{1,G}(\Omega)$ as the set of all measurable functions $h$ such that $h$ and all its distributional derivatives $D_{i}h$, $i=1,\ldots,n$, belong to $L^{G}(\Omega)$. In this case, we endow $W^{1,G}(\Omega)$ with the norm
\[
\|h\|_{W^{1,G}(\Omega)}=\|h\|_{L^{G}(\Omega)}+\|\nabla h\|_{L^{G}(\Omega)}.
\]
\end{definition}

\begin{remark}
We observe that, under assumptions {\bf (PC)--(QC)}, $L^{G}(\Omega)$ and $W^{1,G}(\Omega)$ are reflexive Banach spaces (see for instance \cite{HarHas19}).
\end{remark}

\begin{definition}
A function $u\in W^{1,G}_{\mathrm{loc}}(\Omega)$ is called a weak supersolution
(respectively, subsolution) to $\Delta_gu=f$ if, for every nonnegative test
function $\phi\in C_0^\infty(\Omega)$, the following inequality holds:
\begin{equation}\label{test}
\int_\Omega \frac{g(|\nabla u|)}{|\nabla u|}\nabla u\cdot \nabla\phi\,dx
\geq
-\int_\Omega f\phi\,dx
\qquad
(\text{respectively, } \leq).
\end{equation}

We say that $u$ is a weak solution to $\Delta_gu=f$ if
\[
\int_\Omega \frac{g(|\nabla u|)}{|\nabla u|}\nabla u\cdot \nabla\phi\,dx
=
-\int_\Omega f\phi\,dx,
\]
for all $\phi\in C_0^\infty(\Omega)$. Whe $f\equiv0$ we say that $u$ is g-harmonic.
\end{definition}

\begin{remark}
    One can show by a density argument that test functions can be taken in $0\leq \phi\in W^{1,G}_c(\Omega),$ which denotes the functions in $W^{1,G}(\Omega)$ that have compact support (see \cite{BraMorei,HarHas19}).
\end{remark}

\begin{lemma}\label{behaviornorm}
If \(u\in L^{G}(\Omega)\), then there exists a positive constant \(\mathrm{C}=\mathrm{C}(\delta_{0},g_{0})\) such that 
\[
||u||_{L^G(\Omega)}\leq \mathrm{C}\max\left\{\left(\int_{\Omega}{G}(|u|)dx\right)^{\frac{1}{1+\delta_{0}}},\left(\int_{\Omega}{G}(|u|)dx\right)^{\frac{1}{1+g_{0}}}\right\}.
\]
\end{lemma}
\begin{proof}
See \cite[Lemma 2.3]{MarWol08}
\end{proof}

In the next sections, we will need a Poincar\'e-type inequality in the framework of Orlicz--Sobolev spaces, which is summarized in the following result.
\begin{lemma}[\bf Poincar\'e-type inequality]\label{poincareinequality}
Let \(u\in W^{1,G}(\Omega)\) such that \(u=0\) on \(\partial \Omega\), then
\[
\int_{\Omega}{G}\left(\frac{|u|}{R}\right)\leq \int_{\Omega}G(|\nabla u|)\quad \text{for}\,\, R=\operatorname{diam}(\Omega).
\]
\end{lemma}
\begin{proof}
See \cite[Lemma 2.4]{MarWol08}.
\end{proof}

As mentioned, the lack of homogeneity is a serious technical difficulty in
our context. The following lemmas gives the alternatives that hold in this scenario.
\begin{lemma}\label{lemm1}
    Let $G$ be a N-function that satisfies {\bf (PC)-(QC)} and assume that $1<\delta_0\leq g_0$ and that $g$ is convex. Then there exist a positive constante $C=C(\delta_0)$ such that for all $a, b\in \mathbb{R}^n$
    $$\left( \frac{g(|a|)}{|a|}a- \frac{g(|b|)}{|b|}b\right)\cdot(a-b)\geq CG(|a-b|).  $$
\end{lemma}
\begin{proof}
    See \cite[Lemma 3.1]{Cant}.
\end{proof}

\begin{lemma}
Let $G$ be a N-function that satisfies {\bf (PC)-(QC)}. Then, we have the following properties:

\begin{itemize}
\item[(g1)] $\min\{s^\delta, s^{g_0}\}g(t) \leq g(st) \leq \max\{s^\delta, s^{g_0}\}g(t),$
\item[(g2)] $G$ is convex and $C^2,$
\item[(g3)] $\dfrac{t g(t)}{1 + g_0} \leq G(t) \leq t g(t) \quad \forall t > 0.$
\item[(G1)] $\min\{s^{\delta+1}, s^{g_0+1}\} \frac{G(t)}{1 + g_0} \leq G(st) \leq (1 + g_0)\max\{s^{\delta+1}, s^{g_0+1}\}G(t)$.
\item[(G2)] $G(a + b) \leq 2^{g_0}(1 + g_0)\big(G(a) + G(b)\big), \quad \forall a,b > 0.$
\end{itemize}
\end{lemma}

\begin{proof}
For the proofs of (g1)--(g3), see \cite{Lieb}. Conditions (g1) and (g3) yield a corresponding inequality for $G$, which implies (G1). Furthermore, by the convexity of $G$ and this inequality, we deduce (G2).
\end{proof}

As $g$ is strictly increasing we can define $g^{-1}$. This inverse function in turn satisfies {\bf (PC)} and a condition similar to {\bf (QC)} (see \cite[Lemma  2.2]{MarWol08}). Moreover if $\tilde{G}$ is such that $\tilde{G}'=g^{-1},$ then we have
\begin{itemize}
    \item[$(\tilde{G}1)$] $\frac{1+\delta_0}{\delta_0}\min\{s^{1+1/\delta_0}, s^{1+1/g_0}\}\tilde{G}(t) \leq \tilde{G}(st)\leq \frac{\delta_0}{1+\delta_0}\max\{s^{1+1/\delta_0}, s^{1+1/g_0}\}\tilde{G}(t);$ 
    \item[$(\tilde{g}1)$] $ab \leq G(a) + \tilde{G}(b)\quad \forall a,b>0.$
\end{itemize}

Lastly, observe that if $A(p)=\frac{g(|p|)}{|p|}p$ and $a^{ij}=\frac{\partial A_i}{\partial p_j},$ then it follows from condition {\bf (QC)} that
 $$\min\{\delta_0, 1\}\frac{g(|p|)}{|p|}|\xi|^2 \leq a^{ij}\xi_i\xi_j \leq \max\{\delta_0, 1\}\frac{g(|p|)}{|p|}|\xi|^2,   $$
which means that the equation $\Delta_gv=0$ is uniformly elliptic for $\frac{g(|p|)}{|p|}$ bounded and bounded away from zero.

\section{Boundary Hölder regularity: Proof of Theorem \ref{thm1}}\label{Sec3}

First, observe that, up to the addition of a constant (which preserves the class of solutions) we may assume without loss of generality that
 $\psi(0)=0.$ Let $K=||u||_{L^\infty(\Omega\cap B_1)}+ [\psi]_{C^\alpha(0)}$ and denote $\Omega_r=\Omega\cap B_r.$ To establish that \(u\in C^\alpha(0)\), it suffices to prove the following statement: there exist universal constants \(\delta,\rho\in(0,1)\) and \(M\ge 1\) such that, for every integer \(k\ge 0\), one has
  \begin{equation}\label{equ3}
      ||u||_{L^\infty(\Omega_{\rho^k})} \leq M K \rho^{k\alpha}.
  \end{equation}
 We prove \eqref{equ3} by induction on \(k\). The case \(k=0\) is immediate. Assume that \eqref{equ3} holds for some \(k\geq 0\). We shall show that it also holds for \(k+1\).

 Set $r=\frac{\rho^k}{2}.$ Since \(\Omega\) is \(\delta\)-Reifenberg flat from the exterior at the origin, there exists a coordinate system \(\{x_1,\ldots,x_n\}\) such that
 $$B_r(0)\cap\Omega\subset B_r(0)\cap\{x_n>-\delta r\}.  $$
 We now introduce the following auxiliary function for $\gamma>0$:
 \begin{equation}
     h(x)=\left(\frac{1}{2}\right)^{-\gamma} - \left\vert x+\left(\delta+ \frac{1}{2}\right)e_n\right\vert^{-\gamma}.
 \end{equation}
We claim that $\Delta_gh \leq 0,$ for $\gamma>0$ large enough. Indeed, if $\lambda=\min\{1,\delta_0\}$ and $\Lambda=\max\{1,g_0\}$
\begin{align*}
    \Delta_g(-|x|^{-\gamma}) &= \sum_{i,j}a^{ij}D_{ij}(-|x|^{-\gamma})\\
    &= \gamma|x|^{-(\gamma+2)}\sum_{k=1}a^{kk} - \gamma(\gamma+2)|x|^{-(\gamma+4)}\sum_{i\neq j}a^{ij}x_jx_i\\
    &\leq \gamma|x|^{-(\gamma+2)}n\Lambda\frac{g(|\nabla(-|x|^{-\gamma})|)}{|\nabla(-|x|^{-\gamma})|} - \gamma(\gamma +2)|x|^{-(\gamma+4)}\lambda\frac{g(|\nabla(-|x|^{-\gamma})|)}{|\nabla(-|x|^{-\gamma})|}|x|^2\\
    &\leq \gamma|x|^{-(\gamma+2)}\frac{g(|\nabla(-|x|^{-\gamma})|)}{|\nabla(-|x|^{-\gamma})|}[n\Lambda - (\gamma+2)\lambda]\leq 0
\end{align*}
provided that \(\gamma\) is chosen so that
 $$\frac{n\Lambda}{\lambda}\leq \gamma +2. $$
Let $T_r:=\{(x',0); |x'|<r\}, \, \tilde{B}_r^+:=B_r^+ - \delta re_n,\,\, \tilde{T}_r:=T_r-\delta re_n$ and $\tilde{\Omega}_r:=\Omega\cap \tilde{B}_r^+.$ Then $\Omega_{r/2}\subset \tilde{\Omega}_r\subset \Omega_{\rho^k}.$ Now we define the following barrier function
 \begin{equation}\label{barrier}
     \tilde{h}(x)= MK\rho^{k\alpha}h\left(\frac{x}{r}\right).
 \end{equation}
We claim that $\tilde{h}$ satisfies:
\[
\begin{cases}
    \Delta_g\tilde{h}\leq 0 & \text{in } \tilde{B}_r^+; \\
    \tilde{h}\geq 0 & \text{in } \tilde{B}_r^+; \\
    \tilde{h}\geq MK\rho^{k\alpha} & \text{on } \partial\tilde{B}_r^+ \setminus \tilde{T}_r; \\
    \tilde{h}(-\delta r e_n) = 0.
\end{cases}
\]
The first inequality has already been established. To verify the third, we note that \(\gamma>1\) and
 $$\partial\tilde{B}_r^+\setminus \tilde{T}_r=\{x; |x+\delta re_n|=r,\,x_n>-\delta r\}$$
then follows from convexity, provided that \(\gamma\) is chosen sufficiently large
$$\left\vert \frac{x}{r}+ \left(\delta+\frac{1}{2}\right)e_n\right\vert^{-\gamma}\leq 2^{-\gamma}\left( 1+\left(\frac{1}{2}\right)^{-\gamma} \right)\leq \left(\frac{1}{2}\right)^{-\gamma}-1.$$
That is, on $\partial\tilde{B}_r^+ \setminus \tilde{T}_r,$ we have
$$\left(\frac{1}{2}\right)^{-\gamma} - \left\vert \frac{x}{r}+ \left(\delta+\frac{1}{2}\right)e_n\right\vert^{-\gamma}\geq 1,  $$
It follows that $\tilde{h}(x)\geq MK\rho^{k\alpha}h\left(\frac{x}{r}\right)\geq MK\rho^{k\alpha}.$
Finally, the second inequality and the fact that $\tilde{h}(-\delta re_n)=0$ are obviously.

Let $$\delta=\rho.$$
It follows from the Mean Value Theorem that 
$$h\left(\frac{x}{r}\right)\leq \tilde{C}\rho \quad\text{in}\quad \{(x',x_n); x'=0,\,-\delta r\leq x_n\leq 2\rho r\},$$
where $\tilde{C}=3\gamma\left(\frac{1}{2}\right)^{-\gamma-1}.$
Thus, we obtain 
$$\tilde{h}\leq MK\rho^{(K+1)\alpha}\tilde{C}\rho^{1-\alpha} \quad\text{in}\quad \{(x',x_n); x'=0,\,-\delta r\leq x_n\leq 2\rho r\}.$$

On $\Omega\cap\partial\tilde{B}_r^+=\Omega\cap(\partial\tilde{B}_r^+\setminus \tilde{T}_r),$ the induction hypothesis gives that
 $$u\leq ||u||_{L^\infty(\Omega_{\rho^k})}\leq MK\rho^{k\alpha} \leq \tilde{h}. $$
On the other hand, on $\partial\Omega\cap\tilde{B}_r^+,$ since $u=\psi$ and $\tilde{h}\geq 0$ we have
$$u\leq \tilde{h}+ ||\psi||_{L^\infty}.$$
Therefore, the Comparison Principle \cite[Lemma 2.8]{MarWol08} yields
$$u(x)\leq \tilde{h}(x) + ||\psi||_{L^\infty(\partial\Omega\cap\tilde{B}_r^+)}\quad\forall x\in \tilde{\Omega}_r. $$
Thus, using the fact that $\psi\in C^\alpha(0),$ we obtain in $\tilde{\Omega}_r\cap\{(x',x_n); x'=0,\,-\delta r\leq x_n\leq 2\rho r\}$ for a constant $C=C(n, \delta_0, g_0)>0$ 
$$u\leq \tilde{h} + ||\psi||_{L^\infty(\partial\Omega\cap\tilde{B}_r^+)} \leq \left(C\rho^{1-\alpha}+ \frac{C}{M\rho^\alpha}\right)MK\rho^{(k+1)\alpha}. $$
We choose \(\rho\in(0,1)\) sufficiently small and M sufficiently large so that
 $$\left(C\rho^{1-\alpha}+\frac{C}{M\rho^\alpha} \right)\leq 1. $$

Then, $u\leq MK\rho^{(k+1)\alpha}$ in the strip $\{(x',x_n); x'=0,\,-\delta r\leq x_n\leq 2\rho r\}.$
By placing suitable translates of the barrier (with $v(x',-\delta r)=0$ for some $x'\in B_{\rho^{k+1}})$ and repeating the above argument, we obtain
$$\sup_{\Omega_{\rho^{k+1}}}u \leq KM\rho^{(k+1)\alpha}. $$
Similarly, one shows that 
 $$\inf_{\Omega_{\rho^{k+1}}}u \geq -KM\rho^{(k+1)\alpha}, $$
which gives  
$$||u||_{L^\infty(\Omega_{\rho^{k+1}})} \leq MK\rho^{(k+1)\alpha}.$$
This completes the induction and hence the proof.


\section{ABP-Type Maximum Principle}\label{Sec4}

We need to put in a little more effort for the corresponding Poisson equation, because the difference between two solutions no longer satisfies any equation. The key tool for overcoming this difficulty is an ABP-type estimate. Prior to its derivation, we establish the following technical lemma, which can be viewed as an improvement of a result from \cite{Lian}. 

\begin{lemma}\label{lemaux}
    Let $\phi:[0,\infty)\to [0,\infty)$ be a nonnegative nonincreasing function.
    Assume that for $h>k$
    $$\phi(h)\leq \frac{c_0}{\min\{(h-k)^{\gamma_1},(h-k)^{\gamma_2}\}}\max\{\phi(k)^{\beta_1},\phi(k)^{\beta_2}\}  $$
    where $c_0>0,\,\,\,0<\gamma_1<\gamma_2$ and $1<\beta_1<\beta_2.$
    Then
   $$\phi(r)=0\quad\text{where}\quad r=\max\left\{ \left(c_0 2^{\gamma_2(1+\frac{1}{\beta_1-1})+2}\phi(0)^\theta\right)^{1/\gamma_1}, \left(c_0 2^{\gamma_2(1+\frac{1}{\beta_1-1})+2}\phi(0)^\theta\right)^{1/\gamma_2} \right\}$$
   and $\theta$ is a constant given by
    $$\theta=
   \begin{cases}
    \beta_2, & \text{if } \phi(0) > 1 \\
    \dfrac{\beta_1-1}{2}, & \text{if } \phi(0) \leq 1.
   \end{cases}
    $$
\end{lemma}

\begin{proof}
   Let $k_m=r\left(1-2^{-m}\right)$ for integers $m\geq 0.$ Since $0<k_m<k_{m+1}\leq r$ and $\phi$ is nonnegative nonincreasing function, it is enough to prove that for any integer $m\geq 0$ 
   $$\phi(k_m)\leq \frac{\phi(0)}{2^{m\eta}}\quad\text{where}\quad \eta=\frac{\gamma_2}{\beta_1-1}+1.  $$
   We prove by induction. For $m=0$ it holds clearly, suppose that it holds
 for $m$ and we need to prove that it holds for $m + 1.$ 
 Since,
  \begin{align*}
      \min\{(k_{m+1}-k_m)^{\gamma_1}, (k_{m+1}-k_m)^{\gamma_2}\} &= \min\{(r2^{-m-1})^{\gamma_1}, (r2^{-m-1})^{\gamma_2}\}\\
      &\geq 2^{-(m+1)\gamma_2}\min\{r^{\gamma_1}, r^{\gamma_2}\},
  \end{align*}
 it follows denoting $\zeta:=\gamma_2\left(1+ \frac{1}{\beta_1-1}\right)+2$
 $$\phi(k_{m+1}) \leq 2^{(m+1)\gamma_2}\frac{c_0}{c_02^\zeta\phi(0)^\theta}\max\{\phi(0)^{\beta_1},\phi(0)^{\beta_2}\}.$$
 Thus,
 \begin{align*}
     \phi(k_{m+1}) &\leq 2^{(m+1)\gamma_2}\cdot 2^{-\zeta}\cdot\phi(0)^{-\theta}\max\left\{\left(\frac{\phi(0)}{2^{m\eta}}\right)^{\beta_1}, \left(\frac{\phi(0)}{2^{m\eta}}\right)^{\beta_2}\right\}\\
     &\leq \frac{\phi(0)}{2^{m\eta}}\cdot 2^{(m+1)\gamma_2-\zeta-m\eta(\beta_1-1)}\cdot\max\{\phi(0)^{\beta_1-\theta-1}, \phi(0)^{\beta_2-\theta-1}\}\\
     &\leq \frac{\phi(0)}{2^{m\eta}}\cdot 2^{-\eta}\\
     &= \frac{\phi(0)}{2^{(m+1)\eta}},
 \end{align*}
 where we use that $2^{(m+1)\gamma_2-\zeta-m\eta(\beta_1-1)+\eta}\leq 1$ and $\max\{\phi(0)^{\beta_1-\theta-1}, \phi(0)^{\beta_1-\theta-1}\}\leq 1.$
\end{proof}

Next, we establish an ABP-type estimate for the difference between solutions.

\begin{proposition}\label{abp}
    Let $G$ be an N-function satisfying {\bf (PC)-(QC)}. Assume that $1<\delta_0\leq g_0$ and that $g$ is convex function. Suppose that $u$ and $v$ are weak solutions of
    $$
   \begin{array}{c c}
    \begin{cases}
        \Delta_g u \geq f_1, & \text{in } \Omega \\[4pt]
        u \leq \psi_1, & \text{on } \partial\Omega
    \end{cases}
    \quad&\quad
    \begin{cases}
        \Delta_g v \leq f_2, & \text{in } \Omega \\[4pt]
        v \geq \psi_2, & \text{on } \partial\Omega
    \end{cases}
\end{array}
    $$
with $\psi_1,\psi_2\in L^\infty(\partial\Omega)$ and $f_1, f_2\in L^q(\Omega)$ for $q>n.$ Then
\begin{align*}
\sup_\Omega(u-v) \leq &\; \|\psi_1-\psi_2\|_{L^\infty(\partial\Omega)} \\
&+ \max \left\{ 
\left( C \max\left\{ 
\|f_1-f_2\|_{L^q}^{1+ \frac{1+1/\delta_0}{1+g_0}}, 
\|f_1-f_2\|_{L^q}^{1+ \frac{1+1/g_0}{1+\delta_0}}
\right\} |\Omega|^\theta \right)^{\frac{1}{1+\delta_0}}, \right. \\
&\qquad\quad \left.
\left( C \max\left\{ 
\|f_1-f_2\|_{L^q}^{1+ \frac{1+1/\delta_0}{1+g_0}}, 
\|f_1-f_2\|_{L^q}^{1+ \frac{1+1/g_0}{1+\delta_0}}
\right\} |\Omega|^\theta \right)^{\frac{1}{1+g_0}} 
\right\},
\end{align*}
where $C=C(n,g_0,\delta_0,G(1),\tilde{G}(1),diam(\Omega))>0$ and $\theta$ is a positive constant given by
 $$\theta=
   \begin{cases}
    1+\left(\frac{1}{n}-\frac{1}{q}\right)\left(1+ \frac{1+1/g_0}{1+\delta_0} \right), & \text{if } |\Omega| > 1 \\
    \frac{1}{2}\left(\frac{1}{n}-\frac{1}{q}\right)\left(1+\frac{1+1/\delta_0}{1+g_0}\right), & \text{if } |\Omega| \leq 1.
   \end{cases}
    $$
    
\end{proposition}

\begin{proof}
  For any $l\geq 0,$ let
  $$A(l)=\{x\in \Omega;\,u(x)-v(x) - ||\psi_1-\psi_2||_{L^\infty}\geq l\} $$
 Given $k\geq 0,$ define
 \[
w(x) = 
\begin{cases}
    u(x) - v(x) - \|\psi_1 - \psi_2\|_{L^\infty} - k, & \text{if } x \in A(k) \\[4pt]
    0, & \text{if } x \in \Omega \setminus A(k)
\end{cases}
\]
Observe that $w\geq 0$ in $\Omega$ and $w=0$ on $\partial\Omega,$ thus using as a test function we get
$$\int_\Omega\frac{g(|\nabla v|)}{|\nabla v|}\nabla v\cdot\nabla w \geq -\int_\Omega f_2w \quad\text{and}\quad \int_\Omega\frac{g(|\nabla u|)}{|\nabla u|}\nabla u\cdot\nabla w \leq -\int_\Omega f_1w. $$
Define
 $$I=\int_{A(k)}\left(\frac{g(|\nabla u|)}{|\nabla u|}\nabla u - \frac{g(|\nabla v|)}{|\nabla v|}\nabla v\right)\cdot\left(\nabla u-\nabla v\right)\,dx. $$
Thus we have
 \begin{align*}
     I &\leq \int_\Omega(f_2-f_1)w\,dx\\
     &\leq \left(\int_{A(k)}|f_2-f_1|^q \right)^{1/q}|A(k)|^{\frac{1}{n}-\frac{1}{q}}\left(\int_{A(k)}w^{\frac{n}{n-1}}\right)^{\frac{n-1}{n}}\\
     &\leq C_1|A(k)|^{\frac{1}{n}-\frac{1}{q}}||f_2-f_1||_{L^q}\int_{A(k)}|\nabla w|\,dx.
 \end{align*}

On the other hand, by Lemma \ref{lemm1} exists a positive universal constante $C_2$ such that
 \begin{equation}\label{equ4}
     \int_{A(k)}G(|\nabla w|) \leq C_2 I.
 \end{equation}

By Young's inequality for N-functions $(\tilde{g}1),$ for any $\epsilon>0$ and $C_3=C_1C_2$
$$ \epsilon|\nabla w|\frac{C_3}{\epsilon} \leq G\left(\epsilon|\nabla w|\right) + \tilde{G}\left(\frac{C_3}{\epsilon}\right). $$
We now apply $(G1)$ and $(\tilde{G}1)$ with suitably chosen values of $\epsilon.$
In the case of $$2(1+g_0)C_3|A(k)|^{\frac{1}{n}-\frac{1}{q}}||f_2-f_1||_{L^q}>1$$ we choose
$$\epsilon= \left(\frac{1}{2(1+g_0)C_3|A(k)|^{\frac{1}{n}-\frac{1}{q}}||f_2-f_1||_{L^q}}\right)^{\frac{1}{1+\delta_0}},$$
and if $2(1+g_0)C_3|A(k)|^{\frac{1}{n}-\frac{1}{q}}||f_2-f_1||_{L^q}\leq 1$ we take
$$\epsilon= \left(\frac{1}{2(1+g_0)C_3|A(k)|^{\frac{1}{n}-\frac{1}{q}}||f_2-f_1||_{L^q}}\right)^{\frac{1}{1+g_0}}.$$
Combining those two cases and (\ref{equ4}), we obtain 

\begin{equation}\label{equ5}
\begin{aligned}
\frac{1}{2}\int_{A(k)} G(|\nabla w|) \leq &\; C_3\tilde{G}(C_3) \max\left\{ 
\|f_1-f_2\|_{L^q}^{1+ \frac{1+1/\delta_0}{1+g_0}}, 
\|f_1-f_2\|_{L^q}^{1+ \frac{1+1/g_0}{1+\delta_0}}
\right\} \\
&\times \max\left\{ |A(k)|^{1+\left(\frac{1}{n}-\frac{1}{q}\right)\left(1+\frac{1+1/\delta_0}{1+g_0} \right)}, |A(k)|^{1+\left(\frac{1}{n}-\frac{1}{q}\right)\left(1+\frac{1+1/g_0}{1+\delta_0} \right)} \right\}.
\end{aligned}
\end{equation}

Now since $w>h-k$ in $A(h)$ for $h>k,$ we get using $(G1)$
$$G(w)>G(h-k)\geq \frac{G(1)}{1+g_0}\min\{(h-k)^{1+\delta_0}, (h-k)^{1+g_0}\}. $$
Thus,
\begin{equation}\label{equ6}
   \frac{G(1)}{1+g_0}\min\{(h-k)^{1+\delta_0}, (h-k)^{1+g_0}\}|A(h)| \leq \int_{A(h)}G(w)\,dx. 
\end{equation}

By Poincaré inequality Lemma \ref{poincareinequality}, denoting $R=diam(\Omega)$ we obtain
\begin{equation}\label{equ7}
\begin{aligned}
 \int_{A(k)}G(w)\,dx &\leq (1+g_0)\min\left \{ \left(\frac{1}{R}\right)^{1+\delta_0}, \left(\frac{1}{R}\right)^{1+g_0}  \right \}^{-1}\int_{A(k)}G\left(\frac{w}{R}\right)\,dx\\
 &\leq (1+g_0)\min\{R^{1+\delta_0}, R^{1+g_0}\}\int_{A(k)}G(|\nabla w|)\,dx
\end{aligned}
\end{equation}
Then combining (\ref{equ5}), (\ref{equ6}) and (\ref{equ7}) we get for $h>k$

\begin{equation}\label{eq:estimate}
\begin{aligned}
|A(h)| \leq &\; \frac{\frac{2(1+g_0)^2}{G(1)}\min\{R^{1+\delta_0}, R^{1+g_0}\} C_3\tilde{G}(C_3)}{\min\{(h-k)^{1+\delta_0}, (h-k)^{1+g_0}\}} \\
&\times \max\left\{ 
\|f_1-f_2\|_{L^q}^{1+ \frac{1+1/\delta_0}{1+g_0}}, 
\|f_1-f_2\|_{L^q}^{1+ \frac{1+1/g_0}{1+\delta_0}}
\right\} \\
&\times \max\left\{ |A(k)|^{1+\left(\frac{1}{n}-\frac{1}{q}\right)\left(1+\frac{1+1/\delta_0}{1+g_0} \right)}, |A(k)|^{1+\left(\frac{1}{n}-\frac{1}{q}\right)\left(1+\frac{1+1/g_0}{1+\delta_0} \right)} \right\}.
\end{aligned}
\end{equation}
Applying Lemma \ref{lemaux} with $\phi(t)=|A(t)|,\,\, \gamma_1=1+g_0,\,\,\gamma_2=1+g_0$
$$c_0= \frac{2(1+g_0)^2}{G(1)}\min\{R^{1+\delta_0}, R^{1+g_0}\} C_3\tilde{G}(C_3)\cdot \max\left\{ 
\|f_1-f_2\|_{L^q}^{1+ \frac{1+1/\delta_0}{1+g_0}}, 
\|f_1-f_2\|_{L^q}^{1+ \frac{1+1/g_0}{1+\delta_0}}
\right\}, $$
$$\beta_1= 1+ \left(\frac{1}{n}-\frac{1}{q}\right)\left(1+\frac{1+1/\delta_0}{1+g_0} \right)\quad\text{and}\quad \beta_2= 1+ \left(\frac{1}{n}-\frac{1}{q}\right)\left(1+\frac{1+1/g_0}{1+\delta_0} \right),  $$
 we get $\phi(r)=0$ for
$$
 r=\max\left\{ \left(c_0 2^{\gamma_2(1+\frac{1}{\beta_1-1})+2}\phi(0)^\theta\right)^{1/\gamma_1}, \left(c_0 2^{\gamma_2(1+\frac{1}{\beta_1-1})+2}\phi(0)^\theta\right)^{1/\gamma_2} \right\}.
$$
Recalling that \(\phi(0)\leq |\Omega|\) and that \(\phi\) is nonnegative and nonincreasing, we readily obtain the desired result.
\end{proof}

\section{Boundary Hölder regularity: Proof of Theorem \ref{thm2}}\label{Sec5}

The proof follows the same strategy as that of Theorem \ref{thm1}. Therefore, we may assume that \(\psi(0)=0\), and then, by defining $K= \left(||u||_{L^{\infty}} + [\psi]_{C^\alpha(0)}+ ||f||_{C_q^{-n/q,(1+g_0)\alpha}}+ 1\right)$ and $\Omega_r:=\Omega\cap B_r,$ It is sufficient to prove that there exist universal constants \(\delta, \rho \in (0,1)\) and \(M \geq 1\) such that, for every integer \(k \geq 0\), one has
$$||u||_{L^\infty(\Omega_{\rho^k})}\leq MK^{\left(1+\frac{1+1/g_0}{1+\delta_0}\right)(1+g_0)^{-1}+1}\rho^{k\alpha}.$$
The case $k=0$ is obviously, assume that it holds for some \(k\). We shall show that it also holds for \(k+1\).
 Set $r=\frac{\rho^k}{2}.$ Since \(\Omega\) is \(\delta\)-Reifenberg flat from the exterior at the origin, there exists a coordinate system \(\{x_1,\ldots,x_n\}\) such that
 $$B_r(0)\cap\Omega\subset B_r(0)\cap\{x_n>-\delta r\}.  $$
We consider the barrier function \(\widetilde{h}\) introduced in \eqref{barrier}. As shown above, and using the same notation, \(\widetilde{h}\) satisfies
\[
\begin{cases}
    \Delta_g\tilde{h}\leq 0 & \text{in } \tilde{B}_r^+; \\
    \tilde{h}\geq 0 & \text{in } \tilde{B}_r^+; \\
    \tilde{h}\geq MK\rho^{k\alpha} & \text{on } \partial\tilde{B}_r^+ \setminus \tilde{T}_r; \\
    \tilde{h}(-\delta r e_n) = 0.
\end{cases}
\]
Then as before take $\delta=\rho$ and
\begin{equation}\label{equ8}
    \tilde{h}\leq MK\rho^{(K+1)\alpha}\tilde{C}\rho^{1-\alpha} \quad\text{in}\quad \{(x',x_n); x'=0,\,-\delta r\leq x_n\leq 2\rho r\}.
\end{equation}
By the ABP estimate proposition \ref{abp}, we have for any $x\in \tilde{\Omega}_r$
\begin{align*}
u(x)-\tilde{h}(x) \leq &\; \|\psi\|_{L^\infty(\partial\Omega\cap \tilde{B}_r^+)}+ ||u||_{L^\infty(\tilde{\Omega}_{r})}\\
&+ \max \left\{ 
\left( C \max\left\{ 
\|f\|_{L^q(\tilde{\Omega}_r)}^{1+ \frac{1+1/\delta_0}{1+g_0}}, 
\|f\|_{L^q(\tilde{\Omega}_r)}^{1+ \frac{1+1/g_0}{1+\delta_0}}
\right\} |\tilde{\Omega}_r|^\theta \right)^{\frac{1}{1+\delta_0}}, \right. \\
&\qquad\quad \left.
\left( C \max\left\{ 
\|f\|_{L^q(\tilde{\Omega}_r)}^{1+ \frac{1+1/\delta_0}{1+g_0}}, 
\|f\|_{L^q(\tilde{\Omega}_r)}^{1+ \frac{1+1/g_0}{1+\delta_0}}
\right\} |\tilde{\Omega}_r|^\theta \right)^{\frac{1}{1+g_0}} 
\right\},
\end{align*}
where $C=C(n,g_0,\delta_0,G(1),\tilde{G}(1),diam(\Omega))>0$ and $\theta$ is a positive constant given by
 $$\theta=
   \begin{cases}
    1+\left(\frac{1}{n}-\frac{1}{q}\right)\left(1+ \frac{1+1/g_0}{1+\delta_0} \right), & \text{if } |\tilde{\Omega}_r| > 1 \\
    \frac{1}{2}\left(\frac{1}{n}-\frac{1}{q}\right)\left(1+\frac{1+1/\delta_0}{1+g_0}\right), & \text{if } |\tilde{\Omega}_r| \leq 1.
   \end{cases}
    $$
Since $\tilde{\Omega}_r\subset \Omega_{\rho^k},$ $r\leq 1$ and $K\geq 1$ we get a constant new $C=C(n,g_0,\delta_0,G(1),\tilde{G}(1))>0$ possibly larger than before, using the induction hypothesis
$$u(x)-\tilde{h}(x) \leq CK^{\left(1+\frac{1+1/g_0}{1+\delta_0}\right)(1+g_0)^{-1}+1}\cdot\left[ (2r)^\alpha + r^{(1+g_0)^{-1}\left(n\theta + (1+g_0)\alpha\left(1+\frac{1+1/\delta_0}{1+g_0} \right) \right)}\right], $$
with  $$\theta= \frac{1}{2}\left(\frac{1}{n}-\frac{1}{q}\right)\left(1+\frac{1+1/\delta_0}{1+g_0}\right),$$ 
and  where we have used the fact that $||f||_{L^q(\tilde{\Omega}_r)}\leq Kr^{(1+g_0)\alpha}$ and  $||\psi||_{L^\infty(\partial\tilde{\Omega}_r)}\leq K(2r)^\alpha.$
Given that
$$(1+g_0)^{-1}\left(n\theta + (1+g_0)\alpha\left(1+\frac{1+1/\delta_0}{1+g_0} \right) \right)>\alpha\quad\text{and}\quad r\leq 1, $$
we obtain for any $x\in \tilde{\Omega}_r$ 
\begin{equation}\label{equ9}
    \begin{aligned}
        u(x)-\tilde{h}(x) &\leq CK^{\left(1+\frac{1+1/g_0}{1+\delta_0}\right)(1+g_0)^{-1}+1}\cdot\left[(2r)^\alpha + r^\alpha \right]\\
        &\leq \frac{C}{M\rho^\alpha}MK^{\left(1+\frac{1+1/g_0}{1+\delta_0}\right)(1+g_0)^{-1}+1}\cdot\rho^{(k+1)\alpha}.
    \end{aligned}
\end{equation}

We choose $\rho$ small enough such that $\tilde{C}\cdot\rho^{1-\alpha}\leq 1/2.$ Next take $M$ large enough such that
$$\frac{C}{M\rho^\alpha}\leq \frac{1}{2}.   $$
Then combining (\ref{equ8}) and (\ref{equ9}) we have in the strip $\{(x',x_n); x'=0,\,-\delta r\leq x_n\leq 2\rho r\}$
\begin{align*}
    u =\tilde{h}+ u-\tilde{h} &\leq MK\rho^{(K+1)\alpha}\tilde{C}\rho^{1-\alpha} + \frac{C}{M\rho^\alpha}MK^{\left(1+\frac{1+1/g_0}{1+\delta_0}\right)(1+g_0)^{-1}+1}\cdot\rho^{(k+1)\alpha}\\
    &\leq MK^{\left(1+\frac{1+1/g_0}{1+\delta_0}\right)(1+g_0)^{-1}+1}\cdot\rho^{(k+1)\alpha}.
\end{align*}
By placing suitable translates of the barrier (with $v(x',-\delta r)=0$ for some $x'\in B_{\rho^{k+1}})$ and repeating the above argument, we obtain
$$\sup_{\Omega_{\rho^{k+1}}}u \leq KM^{\left(1+\frac{1+1/g_0}{1+\delta_0}\right)(1+g_0)^{-1}+1}\cdot\rho^{(k+1)\alpha}. $$
Thus, as we argue in Theorem \ref{thm1} we obtain,
$$||u||_{L^\infty(\Omega_{\rho^k})}\leq MK^{\left(1+\frac{1+1/g_0}{1+\delta_0}\right)(1+g_0)^{-1}+1}\cdot\rho^{(k+1)\alpha},$$
which completes the induction and hence the proof.





\begin{thebibliography}{99}



\bibitem{Lieb} Lieberman, G.{\it The natural generalization of the natural conditions of Ladyzhensaya and Uraltseva for elliptic equations.} Commun. Partial Differ. Equ. 16(2 \& 3), 311–361
(1991)


\bibitem{HarHas19} P. Harjulehto and P.~A. H\"ast\"o, {\it Orlicz spaces and generalized Orlicz spaces}, Lecture Notes in Mathematics, 2236, Springer, Cham, 2019.


\bibitem{Reif} E. R. Reifenberg, {\it Solution of the Plateau Problem for m-dimensional surfaces of varying topological type.}
Acta Math., 104:1–92, 1960.

\bibitem{BraMorei} Braga, J. Ederson M., and Diego Moreira. {\it Classification of Nonnegative $g$-Harmonic Functions in Half-Spaces.} Potential Analysis 55.3 (2021): 369-387.

\bibitem{BraSou25} J.~E.~M. Braga and A.~P. Sousa, {\it Up to the Boundary Gradient Estimates for Nonlinear PDEs and Applications in Free Boundary Problems.} Milan J. Math. {\bf 93} (2025), no.~2, 487--515.

\bibitem{Et} Lian, Yuanyuan, et al.{\it Boundary Hölder regularity for elliptic equations.} Journal de Mathématiques Pures et Appliquées 143 (2020): 311-333.

\bibitem{MarWol08} S.~R. Mart\'inez and N.~I. Wolanski, {\it A minimum problem with free boundary in Orlicz spaces.} Adv. Math. {\bf 218} (2008), no.~6, 1914--1971.

\bibitem{By1} S. Byun and L. Wang, {\it Elliptic equations with BMO nonlinearity in Reifenberg domains.} Adv. Math., 219(6):1937–1971, 2008.

\bibitem{By2} S. Byun and L. Wang, {\it Gradient estimates for elliptic systems in non-smooth domains.} Math. Ann., 341(3):629–650, 2008.

\bibitem{By3} S. Byun, L. Wang, and S. Zhou, {\it Nonlinear elliptic equations with BMO coefficients in Reifenberg domains.} J. Funct. Anal., 250(1):167–196, 2007.

\bibitem{Lian} Lian, Yuanyuan, and Kai Zhang, {\it Boundary H\" older Regularity for Elliptic Equations on Reifenberg Flat Domains.} Manuscripta Math., 177(2):Paper No. 14, 21, 2026.


\bibitem{Mila} E. Milakis and T. Toro, {\it Divergence form operators in Reifenberg flat domains. Math.} Z., 264(1):15–41,
2010.

\bibitem{Wu} Wu, Ruimin, Yinsheng Jiang, and Liyuan Wang, {\it Gradient estimates in generalized Orlicz spaces for quasilinear elliptic equations via extrapolation.} AIMS Mathematics 8.10 (2023): 24153-24161

\bibitem{Prade} Prade, Adriano, {\it Boundary regularity for nonlocal elliptic equations over Reifenberg flat domains.} Nonlinear Analysis 261 (2025): 113908.


\bibitem{Leme3} Lemenant, A., Sire, Y., {\it Boundary regularity for the Poisson equation in Reifenberg-flat domains.} Geometric partial differential equations. Vol. 15. CRM Series. Ed. Norm., Pisa, 2013, pp. 189–209.

\bibitem{Cant} Cantizano, Natali A., Ariel M. Salort, and Juan F. Spedaletti, {\it Continuity of solutions for the $\Delta\phi$-Laplacian operator.} Proceedings of the Royal Society of Edinburgh Section A: Mathematics 151.4 (2021): 1355-1382.








\end{thebibliography}
\end{document}